\documentclass[11pt,leqno]{amsart}
\usepackage{amsmath,epsfig,graphicx,color, float,amssymb}
\usepackage{subcaption}
\usepackage{xcolor}
\usepackage{relsize}
\usepackage{comment}
\usepackage[colorlinks, linkcolor=blue,citecolor=blue]{hyperref}
\usepackage{pdfpages}
\usepackage{graphicx}
\usepackage{mwe}
\usepackage{algorithm}
\usepackage{algorithmicx}
\usepackage{algpseudocode}
\usepackage{dsfont}

\usepackage[margin=1.3in]{geometry}

\numberwithin{table}{section}
\numberwithin{figure}{section}
\numberwithin{equation}{section}

\definecolor{darkblue}{rgb}{.2, 0.2,.8}
\definecolor{darkgreen}{rgb}{0,0.5,0.3}
\definecolor{darkred}{rgb}{.8, .1,.1}

\newcommand{\dd}{\mathrm{d}}

\newtheorem{lemma}{Lemma}[section]
\newtheorem{theorem}[lemma]{Theorem}

\newtheorem{assumption}{Assumption}
\newtheorem{definition}[lemma]{Definition}

\newtheorem{remark}{Remark}[section]

\usepackage{bm}

\allowdisplaybreaks

\begin{document}
\title[Pathwise Convergence of a Modular Simulation Scheme for mH-SDEs]{Pathwise Convergence of a Modular Simulation Scheme for Hybrid SDEs with Memory}

\author{Oscar Peralta}

\begin{abstract}
This work introduces hybrid stochastic differential equations with memory (mH-SDEs), a new class of stochastic systems where  transition rates depend on the joint history of both Euclidean and discrete components. This extends existing hybrid stochastic differential equation models that condition transitions only on the Euclidean process history, enabling richer dependencies such as age-based transitions and self-reinforcing dynamics. For mH-SDEs driven by Lévy processes, we develop the Modular-Poisson algorithm, which employs path-dependent uniformization to generate discrete jumps while advancing the Euclidean evolution between jumps using any established SDE solver as a micro-algorithm. 
The main theoretical contribution establishes pathwise convergence with explicit rates, developing new techniques to control error accumulation across regime changes and bound the probability of process decoupling. The modular design allows practitioners to employ existing, well-studied SDE simulation methods while preserving theoretical convergence properties. This work provides the first rigorous convergence analysis for hybrid systems with joint history dependence under jump-diffusion dynamics.
\end{abstract}
\maketitle

\section{Introduction}
\label{sec:intro}

Hybrid stochastic differential equations (H-SDEs) \cite{yin2009hybrid} model systems where a Euclidean-valued process $X=\{X_t\}_{t\ge 0}$ evolves under dynamics that change according to a discrete-valued process $J=\{J_t\}_{t\ge 0}$ (hereafter referred to as the Euclidean and discrete components, respectively). These models arise naturally in applications ranging from mathematical finance, where market regimes shift based on economic conditions, to engineering control systems with mode-dependent dynamics. The numerical analysis literature for H-SDEs has developed along several directions, organized primarily by the structure of the  transition mechanism. For Markovian regime-switching systems, foundational work established Euler–Maruyama schemes \cite{ym04}, with subsequent extensions handling non-Lipschitz coefficients \cite{myy07} and providing pathwise convergence rates \cite{ny12}. Higher-order methods include Milstein-type schemes \cite{nhny17} and general-order approximations \cite{f17}. When drift and diffusion coefficients exhibit super-linear growth, specialized techniques such as tamed-Euler \cite{nnhy18} and projected-Euler methods \cite{zyx21} become relevant. State-dependent transitions have been addressed through dedicated algorithms \cite{yin2010approximation}, while theoretical advances have extended to systems with countably infinite discrete states \cite{Shao2015, XiZhu2017}. For jump-diffusion driven H-SDEs, researchers have developed finite-difference approaches \cite{ysz05} and Euler–Maruyama extensions \cite{mhm19}.

Path-dependent stochastic models typically follow one of two methodological approaches. The first embeds historical dependence directly within the Euclidean dynamics, yielding stochastic functional differential equations (SFDEs) where drift and diffusion coefficients depend functionally on the trajectory history \cite{Mohammed1984, BYY2016}. The second approach augments standard SDEs with a discrete jump component that responds to the Euclidean process behavior.

Most research on regime-switching diffusions adopts this second framework, developing models where transition rates depend on the Euclidean component $X$ through either its current state \cite{yin2010approximation,AlbrecherPeralta2025} or its historical trajectory \cite{ny16,nguyen2021stability}. Recent theoretical progress includes stability analysis for past-dependent transition systems with countable state spaces \cite{NguyenEtAl2024}. However, existing models impose a fundamental restriction: the discrete component can only respond to the Euclidean  process, not to its own historical evolution.

Here we are interested in models where regime transitions explicitly depend on the system's historical behavior, which we term hybrid SDEs with memory (mH-SDEs). Our mH-SDEs allow the discrete jump process to depend on its own past behavior in addition to the Euclidean component's history. This generalization enables transition rates that respond to the complete joint history of both processes, capturing phenomena such as age-dependent regime changes, where transition propensities depend on sojourn times in previous states, or self-reinforcing dynamics, where past jump patterns influence future transition behavior. While our approach builds upon structural elements from past-dependent regime-switching models \cite{ny16, d18}, the primary challenge lies not in formulating these systems, but in developing computationally tractable simulation methods. The initial motivation for studying this specific class of hybrid systems draws upon the thesis of Allan \cite{Allan2023}. While that work provided a preliminary exploration of the topic, the present paper establishes a more general rigorous framework, addressing technical limitations in the original formulation and extending the validity of the simulation algorithms.

The central contribution of this work is a modular simulation approach that transforms the numerical treatment of complex hybrid systems from a problem-specific task into a systematic methodology. Rather than deriving specialized algorithms for each combination of Euclidean dynamics and  transition rules, we develop a unified methodology that employs existing, well-established SDE simulation schemes as interchangeable components while preserving theoretical convergence properties. This innovation leads to the Modular-Poisson algorithm, which addresses the fundamental challenge posed by mH-SDEs: existing numerical methods do not directly extend to handle the joint history dependence that defines these systems. Our algorithm employs a path-dependent uniformization scheme to generate discrete jump times via a master Poisson process, while evolving the Euclidean component between jumps using any suitable SDE solver (termed a micro-algorithm). We establish pathwise convergence with explicit rates for this design, proving that practitioners can incorporate established SDE methods as micro-algorithms without compromising convergence properties.

Establishing this convergence result requires overcoming significant theoretical challenges that arise from the interaction between continuous approximation errors and discrete regime transitions in path-dependent systems. While existing literature addresses related settings, including hybrid systems with delay \cite{LiMao2020,DongMao2022,LiEtAl2023} and neutral SFDEs with Markovian regime changes \cite{HuEtAl2019}, no prior work has tackled the joint history dependence of the random  transition mechanism that characterizes mH-SDEs. Our convergence analysis develops novel techniques for controlling error accumulation across jump events and bounding the probability of divergence between true and approximated transition processes, providing the first rigorous treatment of this model class.

The remainder of this paper is organized as follows. Section~\ref{sec:model} introduces the mathematical  formulation for mH-SDEs and establishes existence and uniqueness of solutions. Section~\ref{sec:apps} demonstrates the model's versatility through applications spanning finance, insurance, reliability theory, and self-exciting systems. Section~\ref{sec:algo} presents the Modular-Poisson algorithm and proves its pathwise convergence with explicit rates. Section~\ref{sec:conclusion} discusses implications and future research directions. Appendix \ref{app:numerical_examples} provides detailed numerical examples that illustrate each application class discussed in Section~\ref{sec:apps}, demonstrating the algorithm's performance across diverse mH-SDE specifications.

\section{Model Definition}\label{sec:model}

In this section, we develop the mathematical framework for hybrid SDEs with memory (mH-SDEs), establishing the system dynamics and conditions for a unique solution to exist. The construction relies on a path-dependent  transition mechanism driven by Poisson random measures and accommodates general Lévy noise in the Euclidean component. Existence and uniqueness of solutions is established via a recursive concatenation procedure and connect the formulation to a representation with time-dependent intensity matrices.

Let us consider a system that evolves on the state space $E=\mathbb{Z}\times\mathbb{R}^p$, equipped with the product topology. Denote by $D_E(I)$ the Skorokhod space of càdlàg functions from an interval $I\subset\mathbb{R}$ to $E$. The notation $J$, $X$, and $Y=(J,X)$ is used to denote the full stochastic processes, with $J_t$, $X_t$, and $Y_t$ representing their values at time $t$. For $t\in (-\infty,\infty)$, we write $Y_{(-\infty,t]}:=\{Y_s\}_{s\le t}$ for the trajectory on the interval $(-\infty,t]$ and $Y_{(-\infty,t)}:=\{Y_s\}_{s< t}$ for the trajectory on the open interval $(-\infty,t)$. This notation extends naturally to the component processes $J$ and $X$.

\begin{remark}
The space $D_E(I)$ is endowed with the topology induced by the Skorokhod $J_1$ metric on compact subintervals \cite{Whitt2002,Billingsley1999}. This choice serves two purposes: it makes the space of càdlàg trajectories complete (under mild conditions on the state space), unlike the uniform topology, and it provides the natural setting for  jump processes by appropriately handling convergence when jump times may differ across  paths. While the $J_1$ topology provides the proper theoretical setting, our convergence analysis in Section \ref{sec:algo} employs uniform distance bounds on compact intervals. Since uniform convergence implies $J_1$ convergence on compacts, our results apply to the standard Skorokhod $J_1$ topology.
\end{remark}

We work on a complete filtered probability space $(\Omega,\mathcal{F},\{\mathcal{F}_t\}_{t\ge0},\mathbb{P})$ satisfying the usual conditions. The filtration $\{\mathcal{F}_t\}_{t\ge0}$ is the completed, right-continuous augmentation of the natural filtration generated by two independent driving processes: a Poisson random measure $\mathfrak{p}$ on $[0,\infty)\times(0,\infty)$ with intensity measure $\dd t\otimes\dd z$, and a $d$-dimensional Lévy process $L$. Here, $\mathfrak{p}$ has time as its first coordinate and marks as its second coordinate, and is adapted to $\{\mathcal{F}_t\}_{t\ge0}$ in the sense that $\mathfrak{p}([0,s] \times B)$ is $\mathcal{F}_s$-measurable for any $s \geq 0$ and Borel set $B \subset (0,\infty)$. Define $\mathcal{F}_{t-}=\sigma\!\left(\bigcup_{s<t}\mathcal{F}_s\right)$ for $t>0$. Equivalently, $\mathcal{F}_{t-}$ is the $\sigma$-algebra that generates the predictable $\sigma$-field. For a $\mathcal{F}_s$-stopping time $\tau$, we define 
\begin{align*}
\mathcal{F}_\tau &:= \{A \in \mathcal{F} : A \cap \{\tau \leq t\} \in \mathcal{F}_t \text{ for all } t \geq 0\}, \\
\mathcal{F}_{\tau-} &:= \{A \in \mathcal{F} : A \cap \{\tau < t\} \in \mathcal{F}_t \text{ for all } t > 0\}.
\end{align*}

For $t>0$, the discrete process $J$ taking values in $\mathbb{Z}$ and the Euclidean process $X$ in $\mathbb{R}^p$ satisfy the coupled system
\begin{align}
\dd J_t &= \int_{(0,\infty)} h(t,z)\,\mathfrak{p}(\dd t,\dd z), \label{eq:hybrid_J} \\
\dd X_t &= g(J_{t-},X_{t-})\,\dd L_t. \label{eq:hybrid_X}
\end{align}
Notably, we do not impose standard regularity conditions such as Lipschitz continuity or linear growth on the coefficient function $g:\mathbb{Z}\times\mathbb{R}^p\to\mathbb{R}^{p\times d}$. Instead, we only require that for each discrete mode $i\in\mathbb{Z}$, the corresponding mode-specific SDE admits a pathwise unique strong solution for any initial condition (see Assumption~\ref{ass:main_construction} below).

The  transition mechanism is governed by a predictable functional $h(t,z)$ that depends on the complete system history. For each $(t,z)$, the quantity $h(t,z)$ is $\mathcal{F}_{t-}$-measurable and determined by a non-anticipative map $\mathfrak{h}$ according to
\[
h(t,z)=\mathfrak{h}\!\left(t,z,\,Y_{(-\infty,t)}\right),
\]
where $Y_{(-\infty,t)}$ represents the joint process trajectory up to time $t$, viewed as an element of $D_E((-\infty,t))$. A jump of magnitude $h(t,z)$ occurs at time $t$ when the Poisson random measure $\mathfrak{p}$ places an atom at $(t,z)$. Note that if $h(t,z) = 0$, no actual jump occurs and the discrete process remains unchanged.

Define a strong solution to the mH-SDE system as follows. A process $Y=(J,X)$ with (deterministic) initial condition $\eta\in D_E((-\infty,0])$ constitutes a strong solution if:
\begin{enumerate}
   \item For $t\le 0$, we have $Y_t=\eta(t)$.
   \item $Y$ is càdlàg and adapted to $\{\mathcal{F}_t\}_{t\ge0}$.
\item The components $(J,X)$ satisfy the integral forms of \eqref{eq:hybrid_J} and \eqref{eq:hybrid_X} for all $t>0$ almost surely, with $X$ remaining continuous at the jump times of $J$.
\end{enumerate}

Existence and uniqueness are established through a constructive approach that builds the solution  trajectory recursively between jump times. To do so, we impose two conditions. The first regularizes the  transition mechanism to ensure a well-defined concatenation scheme; the second guarantees existence of Euclidean path segments for any starting configuration.

\begin{assumption}\label{ass:main_construction}
There exists $\lambda>0$ such that the jump functional has uniformly compact support in the mark variable. That is, for every $t\ge 0$ and $y\in D_E((-\infty,t))$, we have $\mathfrak{h}(t,z,y)=0$ whenever $z\ge\lambda$. Moreover, for each fixed mode $i\in\mathbb{Z}$, initial value $x\in\mathbb{R}^p$, and start time $s\ge0$, the SDE
\[
\dd Z_t=g(i,Z_{t-})\,\dd (\theta_sL)_t,\qquad Z_0=x,
\]
admits a pathwise unique strong solution, where $(\theta_sL)_t:=L_{s+t}-L_s$ is the time-shifted Lévy process. Denote this solution by $Z^{(i,x,s)}=\{Z^{(i,x,s)}_t\}_{t\ge 0}$.
\end{assumption}

Under Assumption~\ref{ass:main_construction}, the jump dynamics of $J$ are governed by the atoms of $\mathfrak{p}$ within the strip $[0,\infty)\times(0,\lambda)$. Let $\{(T_m,U_m)\}_{m\ge1}$ denote these atoms, ordered by increasing arrival time $T_m$. The associated counting process $N=\{N_t\}_{t\ge 0}$ where
\[
N_t:=\mathfrak{p}\big([0,t]\times(0,\lambda)\big)=\sum_{m=1}^\infty \mathds{1}_{\{T_m\le t\}},\qquad t\ge0,
\]
is a Poisson process with rate $\lambda$. Since $N$ has càdlàg  trajectories with almost surely non-explosive arrival times, we have $T_m\to\infty$ as $m\to\infty$.

The solution $Y=(J,X)$ is constructed recursively on the intervals $[T_{m-1},T_m]$ for $m\ge1$, starting from the initial condition on $(-\infty,0]$ and setting $T_0=0$. At each step $m$, we determine the path segment on $[T_{m-1},T_m]$ based on the history accumulated up to time $T_{m-1}$. The construction proceeds as follows.

Let $i_{m-1}=J_{T_{m-1}}$ and $x_{m-1}=X_{T_{m-1}}$ denote the system state at time $T_{m-1}$. On the interval $[T_{m-1},T_m)$, the discrete component remains constant while the Euclidean component evolves according to the mode-specific dynamics, that is,
\begin{align*}
J_t &= i_{m-1}, \qquad t\in[T_{m-1},T_m), \\
X_t &= x_{m-1} + \int_{T_{m-1}}^t g(i_{m-1},X_{s-})\,\dd L_s, \qquad t\in[T_{m-1},T_m).
\end{align*}
By Assumption~\ref{ass:main_construction}, this Euclidean evolution has a unique solution given by $X_{T_{m-1}+s} = Z^{(i_{m-1},x_{m-1},T_{m-1})}_{s}$ for $s\in[0,T_m-T_{m-1})$.

At the jump time $T_m$, the discrete component potentially transitions while the Euclidean component remains continuous. The jump magnitude depends on the complete history accumulated up to time $T_m$, translating to
\begin{align*}
J_{T_m} &= J_{T_m-} + \mathfrak{h}\!\big(T_m,U_m,\,Y_{(-\infty,T_m)}\big), \\
X_{T_m} &= X_{T_m-}.
\end{align*}
This recursive procedure defines a unique càdlàg process $Y$ on $[0,\infty)$ since the jump times are non-explosive.

\begin{theorem}
Under Assumption \ref{ass:main_construction}, there exists a unique strong solution $Y=(J, X)$ to the system \eqref{eq:hybrid_J}–\eqref{eq:hybrid_X}.
\end{theorem}
\begin{proof}
The recursive construction provides both existence and pathwise uniqueness simultaneously, both of which we verify now. For any realization of the driving processes $(L,\mathfrak{p})$, the sequence of potential jump times and marks $\{(T_m,U_m)\}_{m\ge1}$ is determined. At each step $m$, Assumption~\ref{ass:main_construction} guarantees a unique solution to the Euclidean dynamics on $[T_{m-1},T_m)$, while the jump rule at $T_m$ is deterministic given the history $Y_{(-\infty,T_m)}$. Since we construct càdlàg path segments and the jump times are non-explosive, the limiting  process $Y$ exists on $(-\infty,\infty)$ with càdlàg sample trajectories. Adaptedness follows from the fact that each construction step uses only $\mathcal{F}_{t}$-measurable information.

To verify pathwise uniqueness, suppose $\tilde{Y}$ is another strong solution driven by the same noise and initial condition. By induction, if $Y$ and $\tilde{Y}$ agree on $(-\infty,T_{m-1}]$, then they have identical Euclidean evolution on $[T_{m-1},T_m)$ by uniqueness of the SDE solution, and identical jumps at $T_m$ since the functional $\mathfrak{h}$ depends on the same history. Therefore $Y=\tilde{Y}$ on $(-\infty,\infty)$.
\end{proof}

The jump dynamics in \eqref{eq:hybrid_J} admit an equivalent representation through a path- and time-dependent intensity matrix $Q_t(y)=[q_{ij}(t,y)]_{i,j\in\mathbb{Z}}$ for $y\in D_E((-\infty,t))$. This perspective connects our specification to classical  jump process theory and provides the theoretical foundation for the simulation algorithm in Section \ref{sec:algo}.

Indeed, the jump functional $\mathfrak{h}$ uniquely determines the transition rates of $J$. For a given time $t$ and history $y\in D_E((-\infty,t))$, suppose the discrete component satisfies $J_{t-}=i$. The instantaneous rate of transitioning from state $i$ to state $j\neq i$ equals the Lebesgue measure, denoted by $\mathfrak{m}$, of the mark set that produces the corresponding jump increment,
\[
q_{ij}(t,y)\;=\;\mathfrak{m}\!\left(\{z\in[0,\lambda):\mathfrak{h}(t,z,y)=j-i\}\right).
\]
The diagonal entries satisfy the conservation property $q_{ii}(t,y)=-\sum_{j\neq i}q_{ij}(t,y)$, and we define the total exit rate as $q_i(t,y):=-q_{ii}(t,y)$. When evaluated along the solution  process, we employ the shorthand notation $q_{ij}(t):=q_{ij}\!\big(t,Y_{(-\infty,t)}\big)$ and $Q_t:=Q_t\!\big(Y_{(-\infty,t)}\big)$. By construction and Assumption~\ref{ass:main_construction}, we have $0\le q_i(t)\le\lambda$ for all times and histories.

Conversely, any specification of the intensity matrix family $\{Q_t(y)\}_{t\ge0,y\in D_E((-\infty,t))}$ with $q_i(t,y)\le\lambda$ determines, in a non-unique manner, a corresponding jump functional $\mathfrak{h}$. A canonical $\mathfrak{h}$ may be obtained as follows. Given $J_{t-}=i$ and history $y$, we partition the interval $[0,\lambda)$ into disjoint subsets $\{\Delta_{ij}(t,y)\}_{j\in\mathbb{Z}}$ such that $\mathfrak{m}(\Delta_{ij}(t,y))=q_{ij}(t,y)$ for $j \neq i$ and $\mathfrak{m}(\Delta_{ii}(t,y))=\lambda - q_i(t,y)$. To ensure the correct measurability conditions, we ask that for each $i,j\in\mathbb{Z}$, the mapping $(z,t,y)\mapsto \mathds{1}_{\{z\in\Delta_{ij}(t,y_{(-\infty,t)})\}}$ remains measurable. The jump functional can then be defined by
\[
\mathfrak{h}(t,z,y)=\sum_{j\in\mathbb{Z}}(j-i)\,\mathds{1}_{\{z\in\Delta_{ij}(t,y)\}},
\]
establishing a canonical correspondence between intensity matrix specifications and jump functionals. 

The uniform bound $\lambda$ on the total exit rates enables a uniformization approach where  transitions are decided only at a discrete set of time points rather than continuously. A master Poisson process with rate $\lambda$ generates potential event times $\{T_m\}_{m\ge 1}$. At each event time $T_m$, given the current state $i=J_{T_m-}$ and accumulated history $Y_{(-\infty,T_m)}$, a jump to $j\neq i$ occurs with probability $q_{ij}(T_m)/\lambda$, while the discrete state remains unchanged with probability $1-q_{i}(T_m)/\lambda$. Note that when $q_i(T_m)=0$, no real jump can occur. This uniformization scheme generates a process with the same distribution as the original construction, providing the algorithmic foundation for our simulation approach in Section~\ref{sec:algo}.

\section{Illustrative Applications of the Framework}\label{sec:apps}
The versatility of mH-SDEs lies in their ability to encode diverse memory structures through the time- and trajectory-dependent intensity matrix
\[
Q_t(y)=\big[q_{ij}(t,y)\big]_{i,j\in\mathbb{Z}},\qquad y\in D_E((-\infty,t)).
\]
This section demonstrates how  path functionals of the history $Y_{(-\infty,t)}$ can capture empirically motivated transition mechanisms across different application domains. Our examples illustrate dependencies on the Euclidean  process history, discrete timing patterns, underlying noise characteristics, and aggregate system states. The  approach's ability to specify transition logic via path functionals enables practitioners to incorporate well-established concepts within a unified mathematical structure.

To illustrate the practical implementation of these concepts, we provide detailed numerical examples in the Appendix that demonstrate each class of applications: regime transitions with Euclidean path functionals (Appendix~\ref{app:insurance}), semi-Markov transitions with age dependence (Appendix~\ref{app:reliability}), jump activity-based transitions (Appendix~\ref{app:levy}), and self-referential dynamics with path reinforcement (Appendix~\ref{app:reinforcement}).

For notational convenience when evaluated along the solution process, we write
\[
Q_t:=Q_t\!\big(Y_{(-\infty,t)}\big),\qquad q_{ij}(t):=q_{ij}\!\big(t,Y_{(-\infty,t)}\big),\qquad
q_{ii}(t):=-\sum_{j\ne i}q_{ij}(t).
\]

\subsection{Regime Transitions Driven by Euclidean Path Functionals}\label{subsec:ex1}
In quantitative finance, asset price dynamics often exhibit regime-dependent behavior where volatility and drift parameters change based on historical price patterns. This motivates hybrid models where  transition rates depend on path-dependent risk measures and performance indicators. The intensity matrix $Q_t$ can incorporate such dependencies through functionals that have been extensively studied in derivatives pricing and risk management.

Consider two fundamental path-dependent quantities: the occupation time above a threshold and the maximum drawdown from historical peaks. The occupation time of the  process $X$ above a barrier $b$ over a rolling window $[t-\delta, t)$ for fixed $\delta>0$ is defined as
\[
\mathrm{Occ}_{[t-\delta,t)}^{(b)}(X):=\int_{t-\delta}^t \mathbf{1}_{\{X_{s-}\ge b\}}\dd s.
\]
This statistic plays a central role in pricing barrier and corridor options, with its distribution theory well-developed for various jump-diffusion models \cite{CaiChenWan2010}. The maximum drawdown at time $t$ is given by
\[
\mathrm{DD}_t:=M_t-X_t,\qquad M_t:=\sup_{0\le u\le t}X_u,
\]
representing the largest peak-to-trough decline observed up to time $t$. This measure forms the foundation for drawdown-based portfolio insurance strategies and tail-risk metrics such as Conditional Expected Drawdown \cite{CarrZhangHadjiliadis2011, GoldbergMahmoud2017}.

A natural specification for regime-transition intensity incorporates both measures,
\[
q_{ij}(t)=\big(\alpha_{ij}+\beta_{ij}\,\mathrm{Occ}_{[t-\delta,t)}^{(b)}(X)+\gamma_{ij}\,\mathbf{1}_{\{\mathrm{DD}_{t-}\ge \kappa\}}\big)\vee 0,\qquad i\ne j,
\]
where $\alpha_{ij}, \beta_{ij}, \gamma_{ij}$ are model parameters controlling baseline transition rates and sensitivity to historical features. For instance, exceeding a critical drawdown threshold $\kappa$ might trigger a jump to a high-volatility crisis regime, while sustained occupation above a performance benchmark could indicate transition to a stable growth regime. Since these functionals depend only on $Y_{(-\infty,t)}$, the resulting intensity remains predictable, ensuring consistency with the mH-SDE formulation.

\subsection{Age-Dependent and Semi-Markov Transitions}\label{subsec:ex2}
Many systems exhibit memory effects related to the timing structure of discrete transitions rather than Euclidean state values. A prominent example is age-dependency, where transition propensities depend on the elapsed time since entering the current discrete state. Define the current age or sojourn time as
\[
A_t:=t-\sup\{s<t:J_s\ne J_{s-}\}\quad\text{(with convention }\sup\emptyset=0\text{)}.
\]
Age-dependent transition rates characterize semi-Markov processes, which find extensive application in reliability engineering, survival analysis, and life insurance \cite{LimniosOprisan2001,bladt2025approximations}. In these contexts, failure rates or default probabilities may increase with component age (wear-out effects) or decrease (burn-in stabilization).

Our approach naturally accommodates such patterns by specifying intensities as functions of the current age,
\[
q_{ij}(t)=\lambda_{ij}\!\big(A_t\big),\qquad i\ne j.
\]
The functions $\lambda_{ij}$ can encode various hazard rate profiles, including bounded aging functions such as $\lambda_{ij}(A_t) = \alpha_{ij}(1 - e^{-\gamma_{ij} A_t})$ for parameters $\alpha_{ij} \leq \lambda$ and $\gamma_{ij} > 0$, which provides monotonically increasing hazard rates that asymptotically approach the bound $\alpha_{ij}$. This specification captures aging with physical limits, where failure rates increase initially due to wear but eventually stabilize at a maximum level determined by design constraints. The specification could also be extended to permit dependence on the past of the Euclidean process, combining durational effects with Euclidean state memory. This demonstrates that semi-Markov processes emerge as a natural subclass of mH-SDEs while preserving the unified theoretical structure.

\subsection{Transitions Based on the Euclidean Jump Activity}\label{subsec:ex3}

The approach allows regime transitions based on characteristics of the Euclidean process $X$ itself, enabling models that respond to the statistical properties of the system's own jump behavior. This proves particularly relevant for systems where stability depends on the nature and intensity of internal discontinuous events, beyond what traditional diffusive volatility captures.

Consider systems vulnerable to clusters of large discontinuous events in their state variables but resilient to diffusive noise or small jumps. For a threshold $\varepsilon>0$ and fixed window length $\delta>0$, define the large-jump counting process
\[
N^{(\varepsilon)}_{[t-\delta,t)}:=\#\{s\in[t-\delta,t):\|\Delta X_s\|>\varepsilon\},
\]
which remains finite almost surely for any fixed $\varepsilon>0$. Such thresholded jump statistics appear prominently in high-frequency econometrics for detecting and characterizing jump activity \cite{AitSahaliaJacod2011}.

Transition rates can then depend on recent jump activity,
\[
q_{ij}(t)=\eta_{ij}\!\bigl(N^{(\varepsilon)}_{[t-\delta,t)}\bigr),\qquad i\ne j,
\]
allowing the system to enter defensive or high-alert regimes when experiencing bursts of large internal shocks. This mechanism captures scenarios where the discrete state represents different operational modes, such as normal operation versus crisis management, with transitions triggered by the intensity and magnitude of the system's own discontinuous behavior rather than just their cumulative effect.

\subsection{Self-Referential Dynamics and Path Reinforcement}

The discrete  process $J$ can exhibit self-referential behavior through predictable functionals that depend on its own historical evolution $J_{(-\infty,t)}$. This capability enables modeling of adaptive systems with complex memory effects. Two distinct mechanisms are examined: higher-order Markovian dependence, where transition intensities depend on recent state sequences, and  path reinforcement, where accumulated visitation patterns influence future transition propensities.

\subsubsection{Higher-Order Memory Without State Space Augmentation}

Classical higher-order Markov models require expanding the state space to encode sequence dependence. Our approach avoids this complication by incorporating sequence memory directly through  path functionals. For memory order $k\ge1$, define the recent state history vector
\[
H_t=(J_{t-},J^{(1)}_{t-},\ldots,J^{(k-1)}_{t-})\in\mathbb{Z}^k,
\]
where $J^{(\ell)}_{t-}$ denotes the $\ell$-th most recently visited state before the current time. Since $H_t$ constitutes a predictable functional of the discrete  trajectory $J_{(-\infty,t)}$, no formal state space augmentation is required. Transition intensities can then be specified as
\[
q_{ij}(t)=\Lambda_{ij}\!\big(H_t\big),\qquad i\ne j,
\]
defining a $k$-th order model that operates directly on the original space $\mathbb{Z}$.

This approach offers computational advantages over classical state augmentation methods. Traditional treatments embed higher-order  processes into first-order Markov chains on the expanded space $\mathbb{Z}^k$, leading to exponential growth in computational complexity \cite{ChingFungNg2004, ChingNgFung2008}. In contrast, our  path-functional approach maintains the discrete component on $\mathbb{Z}$ while tracking only the summary vector $H_t$. When a transition $i\to j$ occurs, the update rule $(i,h_1,\ldots,h_{k-1})\to (j,i,h_1,\ldots,h_{k-2})$ efficiently maintains the necessary history. This design preserves compatibility with the Modular-Poisson algorithm while avoiding state space explosion.

\subsubsection{Path Reinforcement and Self-Excitation Mechanisms}\label{subsec:ex4b}

 Path reinforcement models exhibit transition rates that adapt based on accumulated visitation statistics. Define the cumulative occupation time and transition count functionals up to time $t$,
\[
\mathrm{Loc}_i(t)=\int_{0}^{t}\mathds{1}_{\{J_{s-}=i\}}\,\dd s,\qquad
\mathrm{Cnt}_{ij}(t)=\#\{s\in(0,t):J_{s-}=i,\ \Delta J_s=j-i\},
\]
where $\Delta J_s:=J_s-J_{s-}$ denotes the jump increment. These statistics enable two types of reinforcement: vertex reinforcement, where states with large occupation times $\mathrm{Loc}_i(t)$ become increasingly attractive or ``sticky'', and edge reinforcement, where frequently traversed  transitions gain higher propensities through accumulated counts $\mathrm{Cnt}_{ij}(t)$. Such mechanisms appear in vertex-reinforced jump  processes and edge-reinforced random walks \cite{DavisVolkov2002, DavisVolkov2004, SabotTarres2015}, as well as in self-exciting point  processes where recent events temporarily elevate future intensities \cite{Hawkes1971}.

To ensure compatibility with our theoretical  formulation, the reinforcement mechanism must respect the uniform intensity bound $\lambda$ from Assumption~\ref{ass:main_construction}. This is achieved through a softmax parametrization,
\begin{equation}\label{eq:selfref-softmax}
q_{ij}(t)=\lambda\,\frac{\exp\{\vartheta_{ij}(t)\}}{1+\sum_{k\neq i}\exp\{\vartheta_{ik}(t)\}},\qquad
\vartheta_{ij}(t)=a_{ij}+b^{(\mathrm{cnt})}_{ij}\,\mathrm{Cnt}_{ij}(t)-c^{(\mathrm{loc})}_{i}\,\mathrm{Loc}_i(t).
\end{equation}
 The parameters $b^{(\mathrm{cnt})}_{ij}\ge0$ and $c^{(\mathrm{loc})}_{i}\ge0$ control the strength of edge reinforcement and vertex inertia, respectively, while $a_{ij}$ sets baseline log-intensities.

This specification automatically satisfies the required bounds,
\[
\sum_{j\neq i}q_{ij}(t)=\lambda\,\frac{\sum_{j\neq i}\exp\{\vartheta_{ij}(t)\}}{1+\sum_{k\neq i}\exp\{\vartheta_{ik}(t)\}}<\lambda.
\]
The reinforcement statistics update incrementally: $\mathrm{Loc}_i(t)$ grows linearly between jump events, while $\mathrm{Cnt}_{ij}(t)$ increments only at transition times.

\section{The Modular-Poisson Algorithm and Convergence}\label{sec:algo}

This section develops the Modular-Poisson algorithm, a flexible simulation approach for mH-SDEs that proposes a modular separation of discrete and Euclidean dynamics. We present the algorithmic structure, state a pathwise convergent result that validates its functionality, provide a detailed proof, and discuss the practical advantages of this modular design.

\subsection{Algorithmic Structure and Convergence Rate}\label{subsec:algo}

The numerical method constructs an approximating process $\widehat{Y}^{(n)} = (\widehat{J}^{(n)}, \widehat{X}^{(n)})$ by generating  path segments between consecutive events of a master Poisson  process. The construction relies on a generic numerical solver for the Euclidean dynamics, which we call a micro-algorithm. We begin by establishing error measurement conventions and micro-algorithm requirements.


Endow the hybrid state space $E = \mathbb{Z} \times \mathbb{R}^p$ with the norm
\[ \|z\|_E := |j| + \|x\|_{\mathbb{R}^p} \]
for any $z=(j,x) \in E$, where $\|\cdot\|_{\mathbb{R}^p}$ denotes an arbitrary norm in $\mathbb{R}^p$. State differences are computed coordinate-wise: $z_1 - z_2 := (j_1 - j_2, x_1 - x_2)$ for $z_1=(j_1,x_1)$ and $z_2=(j_2,x_2)$. For any trajectory $y \in D_E(I)$ and $t\in (-\infty,\infty)$ where $(-\infty,t]\subseteq I\subseteq (-\infty,\infty)$, we define the pathwise supremum norm
\[ \|y\|_t := \sup_{s \le t}\|y_s\|_E. \]
When context permits, we abbreviate $\|\cdot\|_E$ as $\|\cdot\|$ for notational clarity.


The modular architecture separates discrete jump generation from Euclidean evolution. A master Poisson  process with rate $\lambda$ produces potential jump times of the discrete  process, while the Euclidean component evolves according to mode-specific SDE dynamics between these events. This Euclidean evolution requires numerical approximation via a micro-algorithm. We assume access to a generic numerical solver that approximates the true solution $Z^{(i,x,s)}$ of the mode-$i$ SDE starting from state $x$ at time $s$. The micro-algorithm output is denoted $\widehat{Z}^{(i,x,s)}$, with accuracy requirements formalized as follows. In Subsection~\ref{subsec:implication}, we verify that standard SDE numerical methods satisfy these conditions.

\begin{assumption}\label{ass:micro}
There exist universal constants $a\ge0$, $\gamma>0$, and $n_0\ge 1$ such that, for each discrete mode $i\in\mathbb{Z}$, true initial state $x\in\mathbb{R}^p$, start time $s\ge0$, and any $b\ge 0$, whenever the approximated initial state $\widehat{x} \in \mathbb{R}^p$ satisfies $\|\widehat{x} - x\| \le (\log n)^b n^{-\gamma}$, the micro-algorithm produces a pathwise error bounded by
\[
\mathbb{P}\!\left(\big\|Z^{(i,x,s)}-\widehat{Z}^{(i,\widehat{x},s)}\big\|_t>(\log n)^{a+b}n^{-\gamma}\right)\le K_0(t)\,(\log n)^{-1}
\]
for any $t \ge 0$ and all $n\ge n_0$, where $K_0(t)$ is a finite constant depending only on the time horizon $t$.
\end{assumption}

\begin{remark}
The approximated  process notation $\widehat{Z}^{(i,x,s)}$ suppresses explicit dependence on the discretization parameter $n$ for notational simplicity, though the convergence rate scales explicitly with $n$. This convention applies throughout the algorithmic development.
\end{remark}

Assumption \ref{ass:micro} accommodates the practical reality that micro-algorithms must often start from approximated rather than exact initial conditions due to error accumulation from previous simulation intervals. The condition ensures that small initial perturbations maintain a manageable convergence rate, a crucial property for the modular  approach's theoretical integrity.

We now construct the approximating  process $\widehat{Y}^{(n)}$ recursively over the intervals determined by the master Poisson process arrival times $\{T_m\}_{m \ge 0}$, where $T_0 = 0$. The algorithm begins with an approximated version of the initial history $Y_t = (J_t, X_t)$ for $t \le 0$. Let us assume perfect knowledge of the discrete component's history while allowing approximation error in the Euclidean component, in other words,
\begin{itemize}
   \item  $\widehat{J}_s^{(n)} = J_s$ for all $s \le 0$;
   \item  For some $b_0\ge 0$, the initial error satisfies
   \[ \sup_{s \le 0} \|\widehat{X}_s^{(n)} - X_s\| \le (\log n)^{b_0} n^{-\gamma}.\]
\end{itemize}

Given the approximated path history $\widehat{Y}^{(n)}_{(-\infty, T_m]}$ up to time $T_m$, we determine the process behavior on the next interval $(T_m, T_{m+1}]$ through the following steps.

\begin{enumerate}
   \item \textbf{Euclidean evolution on $(T_m, T_{m+1})$:} The discrete component remains frozen while the Euclidean component evolves via the micro-algorithm,
   \begin{align*}
       \widehat{J}_{T_m+s}^{(n)} &= \widehat{J}_{T_m}^{(n)}, \\
       \widehat{X}_{T_m+s}^{(n)} &= \widehat{Z}_s^{(\widehat{J}_{T_m}^{(n)}, \widehat{X}_{T_m}^{(n)}, T_m)} \quad \text{for } s \in (0, T_{m+1}-T_m).
   \end{align*}
   Recall that the micro-algorithm starts from the approximated state $(\widehat{J}_{T_m}^{(n)}, \widehat{X}_{T_m}^{(n)})$, potentially introducing additional error that propagates forward.

   \item \textbf{Transition decision at $T_{m+1}$:} The discrete component potentially  transitions according to the  path-dependent intensity, computed using the approximated history. Define the effective jump functional
   \[
   \widehat{h}(T_{m+1}, z) := \mathfrak{h}(T_{m+1}, z, \widehat{Y}^{(n)}_{(-\infty, T_{m+1})}),
   \]
   and update the discrete state as
   \[
   \widehat{J}_{T_{m+1}}^{(n)} = \widehat{J}_{T_{m+1}-}^{(n)} + \widehat{h}(T_{m+1}, U_{m+1}),
   \]
   where we recall that $\{(T_\ell,U_\ell)\}_{\ell\ge1}$ denote the atoms of the Poisson random measure $\mathfrak{p}$. The Euclidean component remains continuous at $T_{m+1}$ by establishing $\widehat{X}_{T_{m+1}}^{(n)} = \widehat{X}_{T_{m+1}-}^{(n)}$.
\end{enumerate}

The recursive construction translates into the following event-driven simulation procedure.

\begin{algorithm}[H]
\caption{Modular-Poisson Simulation Algorithm}
\label{alg:modular-poisson}
\begin{algorithmic}[1]
\State \textbf{Input:} Initial path $\widehat{Y}^{(n)}$ on $(-\infty, 0]$, horizon $T$, discretization level $n$
\State \textbf{Output:} Simulated path $\widehat{Y}^{(n)}$ on $[0, T]$

\Statex
\State Generate Poisson times $\{T_1, T_2, \ldots, T_M\}$ with rate $\lambda$ up to time $T$
\State Set $m \gets 1$

\allowbreak
\Statex
\While{$T_m \leq T$}
   \State Get current state: $(i, x) \gets (\widehat{J}^{(n)}_{T_{m-1}}, \widehat{X}^{(n)}_{T_{m-1}})$
   \State Evolve Euclidean: $\widehat{X}^{(n)}_{[T_{m-1}, T_m)} \gets \text{MicroAlgorithm}(i, x, T_{m-1}, T_m-T_{m-1})$
   \State \quad \Comment{MicroAlgorithm(mode, initial state, start time, time length)}
   \State Set discrete: $\widehat{J}^{(n)}_{[T_{m-1}, T_m)} \gets i$
   
   \Statex
   \State Compute rates: $q_{ij} \gets q_{ij}(T_m, \widehat{Y}^{(n)}_{(-\infty, T_m)})$ for $j \neq i$
   \State Set $q_{\text{total}} \gets \sum_{j \neq i} q_{ij}$

   \allowbreak
   \State Generate $U \sim \text{Uniform}[0,\lambda]$
   \If{$U < q_{\text{total}}$}
       \State Sample new state $j^*$ with probabilities $\{q_{ij} / q_{\text{total}}\}_{j \neq i}$
       \State Set $\widehat{J}^{(n)}_{T_m} \gets j^*$
   \Else
       \State Set $\widehat{J}^{(n)}_{T_m} \gets i$
   \EndIf
   
   \State Set $\widehat{X}^{(n)}_{T_m} \gets \widehat{X}^{(n)}_{T_m-}$
   \State $m \gets m + 1$
\EndWhile

\allowbreak
\Statex
\If{$T_{m-1} < T< T_m$}
   \State Get final state: $(i, x) \gets (\widehat{J}^{(n)}_{T_{m-1}}, \widehat{X}^{(n)}_{T_{m-1}})$
   \State Evolve to end: $\widehat{X}^{(n)}_{[T_{m-1}, T]} \gets \text{MicroAlgorithm}(i, x, T_{m-1}, T-T_{m-1})$
   \State Set $\widehat{J}^{(n)}_{[T_{m-1}, T]} \gets i$
\EndIf

\State \Return $\widehat{Y}^{(n)}$
\end{algorithmic}
\end{algorithm}

The algorithm's modular structure allows the micro-algorithm subroutine to be any numerical SDE solver satisfying Assumption~\ref{ass:micro}. This flexibility enables practitioners to select specialized solvers for different regimes or problem characteristics.


Convergence of the Modular-Poisson scheme requires an additional regularity condition that controls how the transition mechanism responds to perturbations in the  path history.

\begin{assumption}\label{ass:logHolder}
The mark space partition $\{\Delta_{ij}(t, y)\}_{j\in\mathbb{Z}}$ associated to $\mathfrak{h}$ varies log-Hölder-continuously with respect to the  path history. Specifically, there exists a constant $K_H>0$ such that for any two  trajectories $y_1, y_2 \in D_E((-\infty,t))$ with $\|y_1-y_2\|_t<1$,
\[
\sup_{i\in\mathbb{Z}} \sum_{j\neq i} \mathfrak{m}\big(\Delta_{ij}(t,y_1) \triangle \Delta_{ij}(t,y_2)\big) \le \frac{K_H}{-\log\big(\|y_1-y_2\|_t\big)},
\]
where $\triangle$ denotes the symmetric difference operation between subsets of $\mathbb{R}$.
\end{assumption}

\begin{remark}
This assumption ensures that the transition intensity matrix $Q_t(y)$ exhibits log-Hölder continuity in the  path argument. The condition requires that both the transition rates $q_{ij}(t,y)$ depend continuously on the history $y$ and that the underlying geometric partition of the mark space varies smoothly. Specifically, it implies that there exists some constant $\tilde{K}_H > 0$ such that for every $t\in (-\infty,\infty)$ and $y_1, y_2 \in D_E((-\infty,t))$,
\[ \sup_{i\in\mathbb{Z}} \sum_{j\in\mathbb{Z}} |q_{ij}(t,y_1) - q_{ij}(t,y_2)| \;\le\; \frac{\tilde{K}_H}{-\log\big(\|y_1-y_2\|_t\big)}. \]
This regularity is essential for controlling the probability that small  path perturbations lead to different jump decisions, which could cause the approximated and true  processes to decouple.
\end{remark}

\begin{theorem}[Pathwise Convergence of the Modular-Poisson Algorithm]\label{theo:rateofconvergence}
Suppose Assumptions~\ref{ass:main_construction}, \ref{ass:micro}, and \ref{ass:logHolder} hold. Let the approximated initial  path $\widehat{Y}^{(n)}$ on $(-\infty, 0]$ satisfy $\widehat{J}_s^{(n)} = J_s$ for all $s \le 0$ and
\[ \sup_{s \le 0} \|\widehat{X}_s^{(n)} - X_s\| \le (\log n)^{b_0} n^{-\gamma} \]
for a constant $b_0\ge 0$. Then for any $T>0$ and $\varepsilon_1 \in (0,\gamma)$,
\[
\lim_{n\rightarrow \infty}\mathbb{P}\!\left(\sup_{0\le t \le T}\|\widehat{Y}^{(n)}_t-Y_t\| > n^{-\gamma+\varepsilon_1}\right)
\;=\; 0.
\]
\end{theorem}

This convergence result addresses fundamental challenges unique to the modular simulation of  path-dependent hybrid systems. The theorem establishes that despite the complex interactions between Euclidean approximation errors and discrete  transition decisions, the overall scheme maintains nearly the same convergence rate.

Note that the proof must overcome two primary sources of error propagation. First, the modular construction creates error accumulation as approximated Euclidean segments are concatenated across  regime change intervals. Each micro-algorithm application starts from an already approximated state, causing local errors to compound over time. The challenge lies in proving that this accumulation remains controlled despite a number of  regime transitions occurring.

Second, and more subtly, the  path-dependent transition mechanism introduces the possibility of  process decoupling. Since  transition decisions for the approximated  process $\widehat{J}^{(n)}$ depend on the perturbed history $\widehat{Y}^{(n)}$ rather than the true history $Y$, small Euclidean approximation errors can potentially trigger different discrete  transitions. Once the discrete components diverge, the  processes follow fundamentally different dynamics, rendering the approximation meaningless. The log-Hölder assumption controls this decoupling probability, but the analysis requires careful tracking of how approximation errors interact with the  transition mechanism across multiple  jump events.

The convergence proof, which we present below, is technically involved and requires developing novel techniques to handle the interplay between these error sources across a random number of regime change events.

\subsection{Proof of Theorem~\ref{theo:rateofconvergence}}

The first part of our convergence analysis centers on controlling the decoupling time, that is, the first moment when the approximated discrete  process $\widehat{J}^{(n)}$ diverges from the true  process $J$. We establish key relationships between this decoupling event and the underlying Poisson structure, then employ the log-Hölder assumption to bound decoupling probabilities.


\begin{definition}
\label{def:decoupling_time}
The decoupling time $\iota^{(n)}$ is the first time when the true and approximated discrete  processes differ,
\[ \iota^{(n)} := \inf\{t \ge 0 : J_t \neq \widehat{J}_t^{(n)}\}. \]
\end{definition}

Since both  processes are piecewise constant with jumps occurring only at master Poisson event times $\{T_m\}_{m\ge1}$, any decoupling must coincide with one of these events. This observation motivates the following characterization.

\begin{lemma}
\label{lem:disagreement_index}
Suppose Assumptions~\ref{ass:main_construction} holds. Define $\kappa$ as the index of the first Poisson event where the jump functionals produce different outcomes,
\[
\kappa := \inf\big\{m \ge 1 : \mathfrak{h}(T_m, U_m, Y_{(-\infty, T_m)}) \neq \mathfrak{h}(T_m, U_m, \widehat{Y}^{(n)}_{(-\infty, T_m)})\big\}.
\]
Then $T_\kappa \le \iota^{(n)}$, providing a lower bound for the decoupling time.
\end{lemma}

\begin{proof}
Let us prove by induction that the event $\{\kappa > m\}$ implies  process agreement up to time $T_m$,
\[
\{\kappa > m\} \subseteq \{\iota^{(n)} > T_m\} \quad \text{for all } m \ge 1.
\]

\textbf{Base case:} At $m=0$, both processes $Y$ and $Y^{(n)}$ start with identical discrete components by construction.

\textbf{Inductive step:} Assume the statement holds for some $m \ge 0$. The event $\{\kappa > m+1\}$ decomposes as
\[
\{\kappa > m+1\} = \{\kappa > m\} \cap \big\{\mathfrak{h}(T_{m+1}, U_{m+1}, Y_{(-\infty, T_{m+1})}) = \mathfrak{h}(T_{m+1}, U_{m+1}, \widehat{Y}^{(n)}_{(-\infty, T_{m+1})})\big\}.
\]

By the inductive hypothesis, on $\{\kappa > m\}$ the processes satisfy $J_s = \widehat{J}^{(n)}_s$ for all $s \le T_m$, which implies $J_{T_{m+1}-} = \widehat{J}^{(n)}_{T_{m+1}-}$. The jump updates at $T_{m+1}$ are
\begin{align*}
   J_{T_{m+1}} &= J_{T_{m+1}-} + \mathfrak{h}(T_{m+1}, U_{m+1}, Y_{(-\infty, T_{m+1})}), \\
   \widehat{J}^{(n)}_{T_{m+1}} &= \widehat{J}^{(n)}_{T_{m+1}-} + \mathfrak{h}(T_{m+1}, U_{m+1}, \widehat{Y}^{(n)}_{(-\infty, T_{m+1})}).
\end{align*}

On $\{\kappa > m+1\}$, both the pre-jump states and the jump increments are identical, hence $J_{T_{m+1}} = \widehat{J}^{(n)}_{T_{m+1}}$. This establishes $\{\kappa > m+1\} \subseteq \{\iota^{(n)} > T_{m+1}\}$, completing the induction.
\end{proof}

The following lemma reduces the analysis of decoupling probabilities to studying functional disagreement at individual Poisson events. The log-Hölder assumption provides the key tool for bounding these probabilities.

\begin{lemma}
\label{lem:decoupling_prob_bound}
Let $t \ge 0$ and $y, \hat{y} \in D_E((-\infty, t))$ be two trajectories with identical discrete components. For $U$ uniformly distributed on $[0, \lambda)$, the disagreement probability satisfies
\[
\mathbb{P}\big(\mathfrak{h}(t, U, y) \neq \mathfrak{h}(t, U, \hat{y})\big) \le \frac{1}{\lambda} \sup_{i\in\mathbb{Z}}\sum_{j\neq i} \mathfrak{m}\Big(\big(\Delta_{ij}(t,y) \triangle \Delta_{ij}(t,\hat{y})\big)\Big).
\]
By Assumption~\ref{ass:logHolder}, this yields the bound
\[
\mathbb{P}\big(\mathfrak{h}(t, U, y) \neq \mathfrak{h}(t, U, \hat{y})\big) \le \frac{K_H}{\lambda}\cdot\frac{1}{-\log\big(\|y-\hat{y}\|_t\big)}.
\]
\end{lemma}

\begin{proof}
Since the discrete components are identical, let $i$ denote their common value at time $t^-$. The disagreement occurs when the uniform mark $U$ falls in a region where the partitions differ. This probability equals
\[
\mathbb{P}\big(\mathfrak{h}(t, U, y) \neq \mathfrak{h}(t, U, \hat{y})\big) = \frac{1}{\lambda} \sum_{j\neq i} \mathfrak{m}\Big(\Delta_{ij}(t,y) \triangle \Delta_{ij}(t,\hat{y})\Big).
\]

Taking the supremum over all possible discrete states $i$ and applying Assumption~\ref{ass:logHolder} directly yields the stated bound.
\end{proof}

With the decoupling structure established, we now prove the two main technical lemmas that control decoupling under small errors and pathwise error accumulation, respectively.

\begin{lemma}
\label{lem:decoupling_while_close_is_unlikely}
Suppose Assumptions~\ref{ass:main_construction}, \ref{ass:micro}, and \ref{ass:logHolder} hold. For any $\varepsilon_1 \in (0,\gamma)$ and fixed time horizon $T$,
\begin{align}\label{eq:aux6}
\lim_{n\rightarrow\infty}\mathbb{P}\Big( \kappa \le N_T, \; N_T \le \lfloor\log\log n\rfloor, \; \sup_{0\le t \le T \wedge \iota^{(n)}} \|Y_t - \widehat{Y}^{(n)}_t\| \le n^{-\gamma+\varepsilon_1} \Big) = 0.
\end{align}
\end{lemma}

\begin{proof}
Let $C_n$ denote the event in the l.h.s. of \eqref{eq:aux6}. Decompose it over the number of Poisson arrivals $\ell$ and the specific decoupling event $k$,
\[
C_n = \bigcup_{\ell=1}^{\lfloor\log\log n\rfloor} \bigcup_{k=1}^{\ell} \Big\{ \kappa=k, \; N_T=\ell, \; \sup_{0\le t \le T \wedge \iota^{(n)}} \|Y_t - \widehat{Y}^{(n)}_t\| \le n^{-\gamma+\varepsilon_1} \Big\}.
\]
By the union bound,
\begin{align}
\mathbb{P}(C_n) \le \sum_{\ell=1}^{\lfloor\log\log n\rfloor} \sum_{k=1}^{\ell} \mathbb{P}\Big( \kappa=k, \; N_T=\ell, \; \sup_{0\le t \le T \wedge \iota^{(n)}} \|Y_t - \widehat{Y}^{(n)}_t\| \le n^{-\gamma+\varepsilon_1} \Big).\label{eq:aux7}
\end{align}

For each term in the double sum, we use the tower property of conditional expectation. The event $\{\kappa=k\}$ implies no disagreement occurred before time $T_k$, so the processes remain coupled up to $T_{k-1}$. The probability can be bounded as
\begin{align*}
&\mathbb{P}\Big( \kappa=k, \; N_T=\ell, \; \sup_{0\le t \le T \wedge \iota^{(n)}} \|Y_t - \widehat{Y}^{(n)}_t\| \le n^{-\gamma+\varepsilon_1} \Big) \\
&\le \mathbb{E}\Big[ \mathbb{P}(\{\mathfrak{h}(T_k, U_k, Y_{(-\infty,T_k)}) \neq \mathfrak{h}(T_k, U_k, \widehat{Y}^{(n)}_{(-\infty,T_k)})\} \mid \mathcal{F}_{T_k-}, T_k) \cdot\mathds{1}_{\{\sup_{0\le t <T_k } \|Y_t - \widehat{Y}^{(n)}_t\| \le n^{-\gamma+\varepsilon_1}\}} \Big],
\end{align*}
where we used that on $\{\kappa = k\}$, we have $\iota^{(n)} \geq T_k$, so the supremum bound up to $T \wedge \iota^{(n)}$ implies the bound up to $T_k-$.

By Lemma~\ref{lem:decoupling_prob_bound}, on the event where the pathwise error is bounded by $n^{-\gamma+\varepsilon_1}$,
\begin{align*}
 &\mathbb{P}(\{\mathfrak{h}(T_k, U_k, Y_{(-\infty,T_k)}) \neq \mathfrak{h}(T_k, U_k, \widehat{Y}^{(n)}_{(-\infty,T_k)})\} \mid \mathcal{F}_{T_k-}, T_k) \cdot\mathds{1}_{\{\sup_{0\le t <T_k } \|Y_t - \widehat{Y}^{(n)}_t\| \le n^{-\gamma+\varepsilon_1}\}} \\
 & \quad \le \frac{K_H}{\lambda} \cdot \frac{1}{-\log(\|Y_{(-\infty,T_k)} - \widehat{Y}^{(n)}_{(-\infty,T_k)}\|_{T_k})} \cdot\mathds{1}_{\{\sup_{0\le t <T_k } \|Y_t - \widehat{Y}^{(n)}_t\| \le n^{-\gamma+\varepsilon_1}\}} \\
 &\quad \le \frac{K_H}{\lambda} \cdot \frac{1}{-\log(n^{-\gamma+\varepsilon_1})} = \frac{K_H}{\lambda(\gamma-\varepsilon_1)\log n}.
\end{align*}

Therefore, each individual summand in the r.h.s. of \eqref{eq:aux7} is bounded by $\frac{K_H}{\lambda(\gamma-\varepsilon_1)\log n}$. Since we sum over at most $\lfloor\log\log n\rfloor^2$ terms, we obtain
\begin{align*}
\mathbb{P}(C_n) &\le \frac{\lfloor\log\log n\rfloor^2 \cdot K_H}{\lambda(\gamma-\varepsilon_1)\log n} \\
&= O\left(\frac{(\log\log n)^2}{\log n}\right) \to 0 \quad \text{as } n \to \infty.
\end{align*}
\end{proof}

The second key component of our convergence analysis controls the pathwise error accumulation in the modular construction. This lemma establishes that despite errors compounding across multiple regime change intervals, the overall approximation quality remains sufficient to prevent decoupling with high probability.

\begin{lemma}
\label{lem:error_accumulation}
Under Assumptions~\ref{ass:main_construction} and \ref{ass:micro}, for any $T>0$ and $\varepsilon_1 \in (0,\gamma)$, suppose the initial conditions satisfy $\|\widehat{Y}^{(n)}_0 - Y_0\| \le (\log n)^{b_0} n^{-\gamma}$ for some $b_0 \ge 0$. Then
\[
\lim_{n\rightarrow\infty} \mathbb{P}\Big(\sup_{0\le t \le T \wedge \iota^{(n)}}\|Y_t - \widehat{Y}^{(n)}_t\| > n^{-\gamma+\varepsilon_1}\Big) = 0.
\]
\end{lemma}

\begin{proof}
Let $M = N_{T \wedge \iota^{(n)}}$ denote the number of Poisson events up to time $T \wedge \iota^{(n)}$. The pathwise error accumulates over at most $M+1$ intervals: the initial interval $[0, T_1 \wedge T \wedge \iota^{(n)}]$ and the subsequent intervals $[T_m \wedge T \wedge \iota^{(n)}, T_{m+1} \wedge T \wedge \iota^{(n)}]$ for $m = 1, \ldots, M$.

On each interval $[T_m \wedge T \wedge \iota^{(n)}, T_{m+1} \wedge T \wedge \iota^{(n)})$, the discrete components are identical (since $t < \iota^{(n)}$), so the error comes entirely from the Euclidean evolution. Heuristically, on the event in which decoupling has not occurred, applying the condition in Assumption~\ref{ass:micro} recursively, leads us to an error for the Euclidean component of the order $(\log n)^{b_m} n^{-\gamma}$ up to time $T_m$, where $b_m = b_0 + ma$. More precisely,
\begin{align*}
&\mathbb{P}\Big(\sup_{T_m \wedge T \wedge \iota^{(n)}\le t < T_{m+1}\wedge T \wedge \iota^{(n)}}\|Y_t - \widehat{Y}^{(n)}_t\| > (\log n)^{b_m}n^{-\gamma}, \\
&\qquad\qquad \sup_{0\le t < T_{m}\wedge T \wedge \iota^{(n)}}\|Y_t - \widehat{Y}^{(n)}_t\| \le (\log n)^{b_{m-1}}n^{-\gamma}\Big)\\
&= \mathbb{E}\Big[\mathbb{P}\Big(\sup_{T_m \wedge T \wedge \iota^{(n)}\le t < T_{m+1}\wedge T \wedge \iota^{(n)}}\|Y_t - \widehat{Y}^{(n)}_t\| > (\log n)^{b_m}n^{-\gamma}\mid \mathcal{F}_{T_m \wedge T\wedge \iota^{(n)}-}, T_m \wedge T\wedge \iota^{(n)}\Big) \\
&\qquad\qquad \cdot \mathds{1}_{\sup_{0\le t < T_{m}\wedge T \wedge \iota^{(n)}}\|Y_t - \widehat{Y}^{(n)}_t\| \le (\log n)^{b_{m-1}}n^{-\gamma}}\Big]\\
&\le \mathbb{E}\Big[K_0(T)(\log n)^{-1} \cdot \mathds{1}_{\sup_{0\le t < T_{m}\wedge T \wedge \iota^{(n)}}\|Y_t - \widehat{Y}^{(n)}_t\| \le (\log n)^{b_{m-1}}n^{-\gamma}}\Big]\\
&\le K_0(T)(\log n)^{-1}.
\end{align*}

The probability can be decomposed as
\begin{align*}
&\mathbb{P}\Big(\sup_{0\le t \le T \wedge \iota^{(n)}}\|Y_t - \widehat{Y}^{(n)}_t\| > n^{-\gamma+\varepsilon_1}\Big)\\
&\le \mathbb{P}\Big(\sup_{0\le t \le T \wedge \iota^{(n)}}\|Y_t - \widehat{Y}^{(n)}_t\| > n^{-\gamma+\varepsilon_1}, M\le \lfloor\log\log n\rfloor\Big)+\mathbb{P}(M> \lfloor\log\log n\rfloor).
\end{align*}

The second summand satisfies $\mathbb{P}(M > \lfloor\log\log n\rfloor) \le \mathbb{P}(N_T > \lfloor\log\log n\rfloor) = o(1)$ as $n \to \infty$. To complete the proof, it remains to investigate the asymptotic properties of the first summand.

Since $(\log n)^{b_0+(\lfloor\log\log n\rfloor+1) a} = o(n^{\varepsilon_1})$ for any $\varepsilon_1 > 0$ (as $\log\log n$ grows slower than any positive power of $\log n$), for sufficiently large $n$ we have
\[
(\log n)^{b_0+(\lfloor\log\log n\rfloor+1) a} n^{-\gamma} < n^{-\gamma+\varepsilon_1}.
\]

Therefore,
\begin{align*}
&\mathbb{P}\Big(\sup_{0\le t \le T \wedge \iota^{(n)}}\|Y_t - \widehat{Y}^{(n)}_t\| > n^{-\gamma+\varepsilon_1}, M\le \lfloor\log\log n\rfloor\Big)\\
&\le \mathbb{P}\Big(\bigcup_{m=0}^{\lfloor\log\log n\rfloor} \Big\{\sup_{T_m \wedge T \wedge \iota^{(n)}\le t < T_{m+1}\wedge T \wedge \iota^{(n)}}\|Y_t - \widehat{Y}^{(n)}_t\| > (\log n)^{b_m}n^{-\gamma}\Big\}\Big)\\
&\le \sum_{m=0}^{\lfloor\log\log n\rfloor}\mathbb{P}\Big(\sup_{T_m \wedge T \wedge \iota^{(n)}\le t < T_{m+1}\wedge T \wedge \iota^{(n)}}\|Y_t - \widehat{Y}^{(n)}_t\| > (\log n)^{b_m}n^{-\gamma}, \\
&\qquad\qquad\qquad\qquad\qquad\qquad \sup_{0\le t < T_{m}\wedge T \wedge \iota^{(n)}}\|Y_t - \widehat{Y}^{(n)}_t\| \le (\log n)^{b_{m-1}}n^{-\gamma}\Big)\\
&\le \sum_{m=0}^{\lfloor\log\log n\rfloor} K_0(T)(\log n)^{-1} = (\lfloor\log\log n\rfloor+1)K_0(T)(\log n)^{-1} = o(1),
\end{align*}
concluding the proof.
\end{proof}

To finalize the proof of Theorem~\ref{theo:rateofconvergence}, we decompose the convergence analysis by controlling the decoupling probability and pathwise error accumulation separately. Let us consider $E_n$ where
\begin{align*}
E_n & := \Big\{\sup_{0\le t \le T}\|\widehat{Y}^{(n)}_t-Y_t\| > n^{-\gamma+\varepsilon_1}\Big\}\cup \big\{\iota^{(n)} \le T\big\}\\
& \subseteq \Big\{\sup_{0\le t \le T \wedge \iota^{(n)}}\|\widehat{Y}^{(n)}_t-Y_t\| > n^{-\gamma+\varepsilon_1}\Big\}\cup \big\{\iota^{(n)} \le T\big\}.
\end{align*}
Employing Lemma~\ref{lem:disagreement_index},
\begin{align}
E_n \subseteq \Big\{\sup_{0\le t \le T \wedge \iota^{(n)}}\|\widehat{Y}^{(n)}_t-Y_t\| > \,n^{-\gamma+\varepsilon_1}\Big\}\cup \{\kappa \le N_T\}.\label{eq:aux_en}
\end{align}

Bound the set in the r.h.s. of \ref{eq:aux_en} by the union of three disjoint sets,
\[
\Big\{\sup_{0\le t \le T \wedge \iota^{(n)}}\|\widehat{Y}^{(n)}_t-Y_t\| > \,n^{-\gamma+\varepsilon_1}\Big\}\cup \{\kappa \le N_T\}\subseteq A^1_n \cup A^2_n \cup A^3_n,
\]
where
\begin{align*}
  A^1_n &:= \Big\{ \kappa \le N_T, \; N_T \le \lfloor\log\log n\rfloor, \; \sup_{0\le t \le T \wedge \iota^{(n)}} \|Y_t - \widehat{Y}^{(n)}_t\| \le n^{-\gamma+\varepsilon_1} \Big\}, \\
  A^2_n &:= \big\{N_T > \lfloor\log\log n\rfloor\big\}, \\
  A^3_n &:= \Big\{\sup_{0\le t \le T \wedge \iota^{(n)}}\|Y_t - \widehat{Y}^{(n)}_t\| > n^{-\gamma+\varepsilon_1}\Big\}.
\end{align*}

The term $A^1_n$ represents decoupling when both the number of Poisson events and pathwise error remain small. By Lemma~\ref{lem:decoupling_while_close_is_unlikely}, $\mathbb{P}(A^1_n) = o(1)$, which satisfies the required bound for sufficiently large $n$. The event $A^2_n$ corresponds to the case in which the Poisson process $N_T$ exceeds $\lfloor\log\log n\rfloor$ events by time $T$, the probability of which is clearly an $o(1)$ term. Finally, the probability of the event $A^3_n$ is $o(1)$ by Lemma~\ref{lem:error_accumulation}. Combining all terms via the union bound, we obtain
\[
\mathbb{P}(E_n) \le \mathbb{P}(A^1_n) + \mathbb{P}(A^2_n) + \mathbb{P}(A^3_n) = o(1),
\]
completing the proof of Theorem~\ref{theo:rateofconvergence}.

\subsection{Advantages of the Simulation Approach}\label{subsec:implication}

The primary advantage of the Modular-Poisson  methodology is its compositional design. The algorithm manages the  path-dependent  transitions at a high level, while allowing the practitioner to plug in any suitable numerical solver, a micro-algorithm, to handle the Euclidean evolution between jumps. This design offers significant flexibility, as one can mix and match different solvers for different regimes. For example, a simple and efficient Euler scheme could be used for a regime with well-behaved coefficients, while a more robust scheme could be employed for complex regimes.

Our main convergence result, Theorem~\ref{theo:rateofconvergence}, relies only on the pathwise convergence in probability stated in Assumption~\ref{ass:micro}. A number standard SDE solvers for Brownian- and Lévy-driven SDEs provide convergence properties that are stronger than required by Assumption \ref{ass:micro}, making them readily compatible with our  approach. Below we demonstrate how standard results from the literature directly imply our micro-algorithm conditions, making a wide range of existing solvers immediately applicable.

\textbf{Strong convergence with polynomial tail bounds.} Many numerical schemes achieve strong pathwise error bounds with polynomial decay rates. Under appropriate regularity conditions, these methods satisfy
\[
\mathbb{P}\!\left(\big\|Z^{(i,x,s)}-\widehat{Z}^{(i,\widehat{x},s)}\big\|_t>(\log n)^{a+b}n^{-\gamma}\right)\le C(t) n^{-q}
\]
for some $q>0$, where the initial error $\|\widehat{x} - x\|$ is of order $(\log n)^b n^{-\gamma}$. The preservation of convergence rates from perturbed inputs to solution outputs via Wong-Zakai methods is established in \cite{RomischWakolbinger1985,nguyen2021wong}. When $q>1$, setting $K_0(t) = C(t)(\log n)^{q-1}$ immediately verifies Assumption~\ref{ass:micro}.

\textbf{Moment-based convergence rates.} A broader class of methods provides $L^p$ error bounds. If the micro-algorithm satisfies
\[
\mathbb{E}\!\left[\big\|Z^{(i,x,s)}-\widehat{Z}^{(i,\widehat{x},s)}\big\|_t^{\,p}\right] \le C(t)(\log n)^{-\alpha} n^{-p\gamma}
\]
for some $p \ge 1$ and $\alpha>0$, then Markov's inequality yields
\[
\mathbb{P}\!\left(\big\|Z^{(i,x,s)}-\widehat{Z}^{(i,\widehat{x},s)}\big\|_t>(\log n)^{a+b}n^{-\gamma}\right) \le \frac{C(t)(\log n)^{-\alpha}}{(\log n)^{p(a+b)}} = C(t)(\log n)^{-\alpha-p(a+b)}.
\]
Choosing $a$ such that $\alpha + p(a+b) \ge 1$ directly verifies Assumption~\ref{ass:micro}. Such $L^p$ strong error estimates are available for Euler-type schemes under global Lipschitz assumptions, as documented in \cite{KuehnSchilling2019}. 

The modular architecture creates immediate compatibility with the extensive toolkit developed for Lévy-driven SDEs. Rather than developing specialized hybrid schemes from scratch, practitioners can draw on existing research on SDE simulation schemes and transfer their convergence properties into the mH-SDE  formulation. Classical Euler schemes provide foundational micro-algorithms, with the original weak convergence analysis by Protter and Talay \cite{ProtterTalay1997} and subsequent strong convergence refinements \cite{KuehnSchilling2019} establishing the pathwise error controls required by Assumption~\ref{ass:micro}. Specialized jump handling addresses the challenges of infinite activity Lévy  processes, often through truncation of small jumps \cite{Rubenthaler2003,fournier2011simulation} or Gaussian approximation of the compensated small-jump component \cite{AsmussenRosinski2001}. Both techniques yield schemes with explicit error bounds that integrate seamlessly into our  approach.

\section{Conclusion}\label{sec:conclusion}

In this work, we introduced hybrid stochastic differential equations with memory (mH-SDEs), a new class of hybrid systems where  transition rates depend on the joint  path history of both the Euclidean and discrete components. For this previously unstudied class of models driven by Lévy  processes, we developed the Modular-Poisson algorithm, a flexible and provably convergent simulation  approach. Our main contribution is a unified pathwise convergence analysis that establishes explicit rates, providing a rigorous foundation for the simulation of these complex systems. The  methodology's key innovation is its modularity, which separates the generation of discrete jumps from the Euclidean evolution, allowing practitioners to incorporate any suitable, well-studied SDE solver as a micro-algorithm without needing to re-derive the convergence theory for the entire hybrid system.

The theoretical foundation of our approach rests critically on the uniform bound assumption (Assumption~\ref{ass:main_construction}), which requires that the total exit rate from any discrete state never exceeds the constant $\lambda$. This seemingly technical condition is fundamental to the algorithm's success, as it enables the uniformization technique that drives our simulation scheme. Without this bound, the Poisson thinning  procedure would fail, and the delicate balance between Euclidean approximation errors (Assumption~\ref{ass:micro}) and discrete decoupling probabilities, controlled through our log-Hölder considerations (Assumption~\ref{ass:logHolder}), would likely break down. The uniform bound is what allows us to prove convergence under such general conditions on the micro-algorithms, making the  approach broadly applicable while maintaining theoretical rigor.

This research opens several promising avenues for future work. From a theoretical perspective, the most significant challenge would be to extend the  methodology beyond uniformizable systems by developing alternative simulation techniques for unbounded intensity rates. This would require fundamentally new approaches to controlling the interaction between Euclidean errors and discrete transition decisions. Furthermore, while our analysis establishes pathwise convergence, investigating weak convergence properties under a more robust structure would be valuable for applications requiring distributional approximations.

\appendix

\section{Numerical Examples}\label{app:numerical_examples}

This appendix demonstrates the Modular-Poisson algorithm through four representative applications that illustrate the main classes of mH-SDEs introduced in Section~\ref{sec:apps}. To ease interpretation, all examples consider binary discrete state spaces $\{0,1\}$, one-dimensional Euclidean  processes, and initial conditions specified only at $t=0$ rather than full  path histories on $(-\infty,0]$. These simplifications are made purely for pedagogical clarity; the framework readily accommodates higher-dimensional state spaces, multi-dimensional SDEs, and complex initial  path specifications. Each example validates the algorithm's performance while highlighting different aspects of  path-dependent  transition mechanisms.

\subsection{Insurance Reserve Dynamics}\label{app:insurance}

We implement the regime-changing insurance employing the model from Section~\ref{sec:apps} where  transitions depend on regulatory compliance history and reserve depletion patterns. The reserve  process $X_t$ follows
\begin{align*}
\dd X_t = \begin{pmatrix} \mu(J_t,X_t) & \sigma(J_t,X_t)\end{pmatrix} \dd L_t = \mu(J_t,X_t)\,\dd t + \sigma(J_t,X_t)\,\dd W_t
\end{align*}
where $L=\{L_t\}_{t\ge 0} = \{(t, W_t)^\intercal\}_{t\ge 0}$ is the two-dimensional Lévy  process with $W=\{W_t\}_{t\ge 0}$ a standard Brownian motion. The coefficient functions are given by normal market dynamics $\mu(0,x) = 0.08x + 0.02$, $\sigma(0,x) = 0.08$ and stressed market dynamics $\mu(1,x) = -0.03x + 0.01$, $\sigma(1,x) = 0.20$.

The  transition mechanism incorporates the occupation time above regulatory minimum $\mathrm{Occ}_{[t-1,t)}^{(1.0)}(X) = \int_{t-1}^t \mathbf{1}_{\{X_{s-} \geq 1.0\}}\,\dd s$ and drawdown $\mathrm{DD}_t = \max_{0 \leq s \leq t} X_s - X_t$ with transition rates
\[
q_{ij}(t) = \big(\alpha_{ij} + \beta_{ij} \,\mathrm{Occ}_{[t-1,t)}^{(1.0)}(X) + \gamma_{ij} \mathbf{1}_{\{\mathrm{DD}_{t-} \geq 0.25\}}\big)\vee 0,\quad i\neq j.
\]
Parameters encode realistic insurance dynamics: $(\alpha_{01}, \beta_{01}, \gamma_{01}) = (0.2, -0.5, 3.0)$ for normal-to-stressed  transitions and $(\alpha_{10}, \beta_{10}, \gamma_{10}) = (0.3, 2.0, -2.0)$ for recovery.

\begin{figure}[H]
\centering
\includegraphics[width=0.6\textwidth]{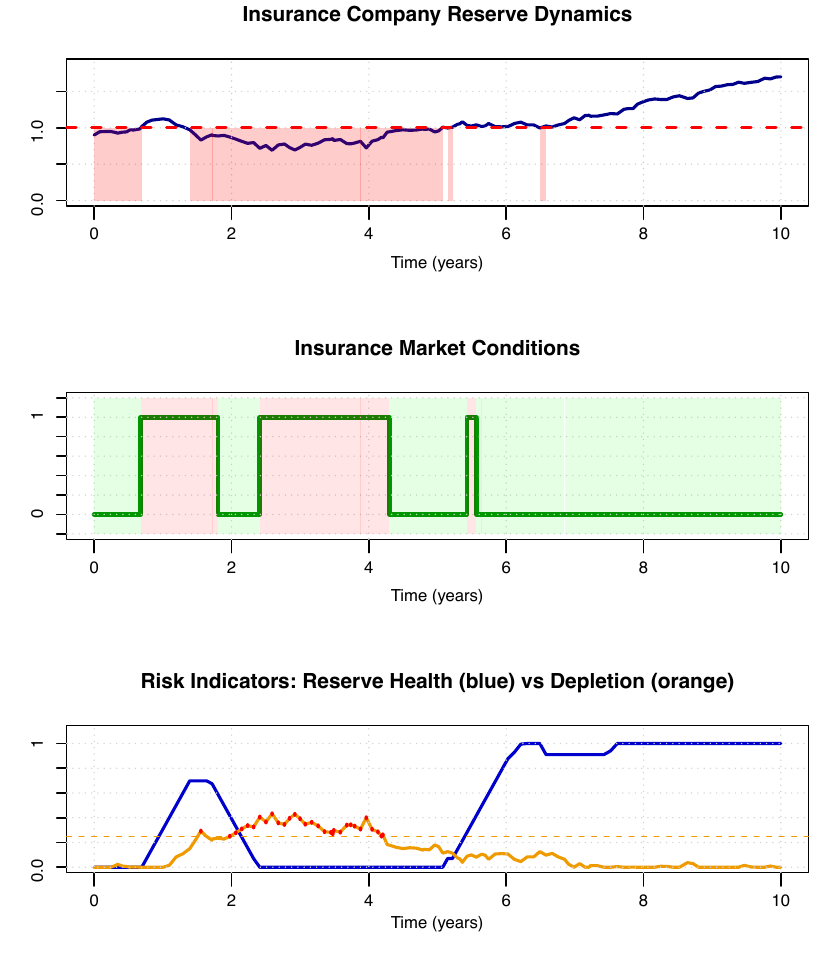}
\caption{Insurance reserve simulation over 10 years starting from $X_0 = 0.9$ (below regulatory minimum). Top: reserve levels with regulatory threshold (red line) and breach periods (shaded). Middle: market regime evolution showing 6  transitions. Bottom: risk indicators driving  transitions.}
\label{fig:insurance}
\end{figure}

Figure~\ref{fig:insurance} demonstrates the  path-dependent  transition mechanism with reserves fluctuating around the regulatory threshold throughout the simulation period. The system exhibits realistic clustering behavior where rapid  transitions occur during periods of reserve instability, followed by stable periods once adequate reserves accumulate.

\subsection{Semi-Markov Reliability System}\label{app:reliability}

This example illustrates the age-dependent  transitions from Section~\ref{sec:apps} through a two-component reliability system where failure rates depend solely on component age, creating non-exponential sojourn distributions.

Component replacement intensities follow bounded hazard functions of age $A_t$. Component 0 exhibits a bathtub curve while Component 1 shows linear wear:
\begin{align}
q_{01}(a) &= \begin{cases}
\min(2.0, 1.5 e^{-3a}) & \text{if } a < 0.5 \\
0.2 & \text{if } 0.5 \leq a < 5.0 \\
\min(2.0, 0.2 + 0.3(a-5)) & \text{if } a \geq 5.0
\end{cases} \\
q_{10}(a) &= \min(2.0, 0.3 + 0.25a)
\end{align}
Since the transition mechanism depends solely on the discrete  process history (component age), the Euclidean dynamics can be specified independently. For instance, the Euclidean process can degrade continuously according to mode-specific drift functions, with component replacement resetting performance to nominal levels while  changing the active component.

\begin{figure}[H]
\centering
\includegraphics[width=0.6\textwidth]{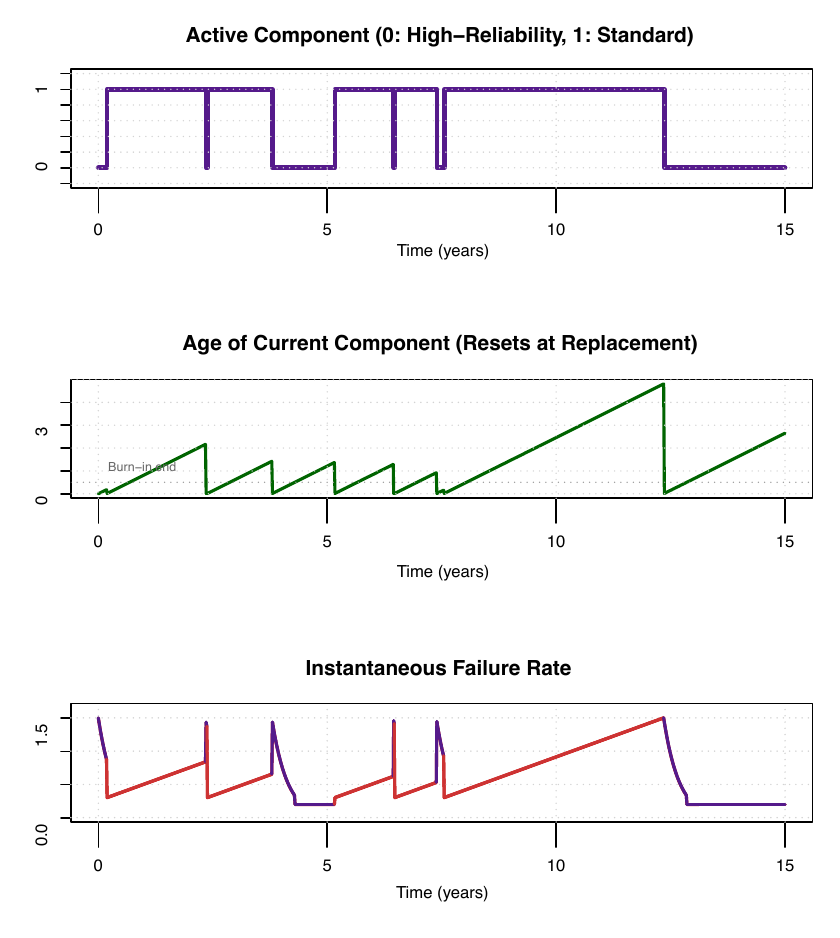}
\caption{Semi-Markov reliability system over 15 years with 10 component replacements. Top: active component transitions between high-reliability (State 0) and standard (State 1) components. Middle: component age evolution showing sawtooth pattern with resets at replacements and key age thresholds. Bottom: instantaneous failure rates displaying bathtub curve behavior for component 0 (purple) and linear wear for component 1 (red), all bounded by 2.0.}
\label{fig:reliability}
\end{figure}

Figure~\ref{fig:reliability} shows the characteristic semi-Markov behavior with varying sojourn times driven by component age rather than performance levels. The high-reliability component (0) dominates system usage while the standard component (1) experiences more frequent early failures.

\subsection{Lévy-Driven Financial Model}\label{app:levy}

We demonstrate jump-based regime changes based on catastrophic price movements, which is a modification of the approach presented in Section~\ref{sec:apps}. Asset price $X_t$ follows
\[
dX_t = \begin{pmatrix} \mu(J_t,X_t) & \sigma_1 (J_t,X_t) & \sigma_2 (J_t,X_t)\end{pmatrix} \dd L_t = \mu(J_t,X_t)\,dt + \sigma\,dW_t + X_{t-}\,dZ_t
\]
where $L=\{L_t\}_{t\ge 0} = \{(t, W_t, Z_t)\}_{t\ge 0}$ is a three-dimensional Lévy process with $W=\{W_t\}_{t\ge 0}$ being a Brownian motion and $Z=\{Z_t\}_{t\ge 0}$ an independent compound Poisson  process with asymmetric double exponential distribution. The coefficient functions are bull market drift $\mu(0,x) = 0.15x$, bear market drift $\mu(1,x) = -0.10x$, and constant diffusion $\sigma = 1$.

 Regime changes depend on catastrophic relative price movements exceeding 15\% within rolling windows. Define the signed jump counting processes:
\begin{align*}
N^+_{[t-1,t)} &:= \#\left\{s \in [t-1,t) : \frac{X_s - X_{s-}}{X_{s-}} > 0.15\right\} \\
N^-_{[t-1,t)} &:= \#\left\{s \in [t-1,t) : \frac{X_s - X_{s-}}{X_{s-}} < -0.15\right\}
\end{align*}
where $(X_s - X_{s-})/X_{s-}$ represents the relative price jump at time $s$. These count catastrophic rallies (price increases exceeding 15\%) and crashes (price drops exceeding 15\%) respectively. The transition rates are
\begin{align*}
q_{01}(t) &= \min\big(2.0,\, 0.1 + 0.8 N^-_{[t-1,t)}\big) \\
q_{10}(t) &= \min\big(2.0,\, 0.1 + 0.6 N^+_{[t-1,t)}\big)
\end{align*}

\begin{figure}[H]
\centering
\includegraphics[width=0.6\textwidth]{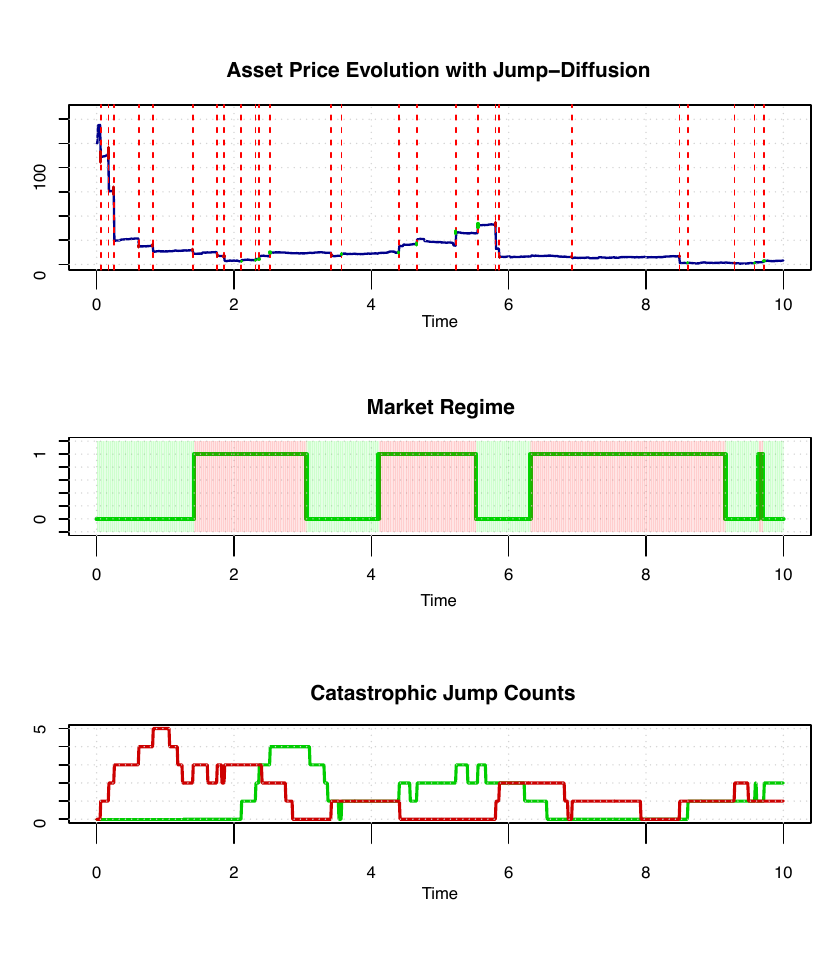}
\caption{Lévy-driven financial model over 10 years showing 8 regime changes. Top: asset price evolution with dramatic discontinuous jumps (vertical red lines) driven by the compound Poisson process. Middle: market regime  transitions between bull (State 0) and bear (State 1) markets. Bottom: rolling window counts of catastrophic positive jumps (green) and negative jumps (red) exceeding 15\% thresholds that drive regime  changes.}
\label{fig:levy}
\end{figure}

Figure~\ref{fig:levy} captures empirical market dynamics where crash clusters precipitate bear markets and rally periods signal recoveries. The asymmetric response parameters reflect the market tendency to decline rapidly but recover gradually.

\subsection{Path Reinforcement in Production Systems}\label{app:reinforcement}

This example demonstrates the self-referential dynamics from Section~\ref{sec:apps} where transition propensities evolve based on accumulated  transition history. The production system operates with two modes having different efficiency targets, with  transitions governed by the softmax parametrization from equation~\eqref{eq:selfref-softmax}.

The transition rates incorporate occupation times $\mathrm{Loc}_i(t)$ and transition counts $\mathrm{Cnt}_{ij}(t)$ through the  path-dependent log-intensities:
\begin{align*}
\theta_{01}(t) &= 0.0 + 0.25 \cdot \mathrm{Cnt}_{01}(t) - 0.005 \cdot \mathrm{Loc}_0(t) \\
\theta_{10}(t) &= 0.2 + 0.20 \cdot \mathrm{Cnt}_{10}(t) - 0.008 \cdot \mathrm{Loc}_1(t)
\end{align*}

The softmax transformation ensures rates remain bounded by $\lambda = 2.0$:
\[
q_{ij}(t) = \lambda \cdot \frac{\exp\{\theta_{ij}(t)\}}{1 + \sum_{k \neq i}\exp\{\theta_{ik}(t)\}}
\]

This specification captures organizational learning through edge reinforcement (repeated  transitions become easier) and lock-in effects through vertex inertia (extended stays create mild resistance to  changing). Since transitions depend solely on the discrete  process history, no specific Euclidean dynamics are required for this demonstration.

\begin{figure}[H]
\centering
\includegraphics[width=0.6\textwidth]{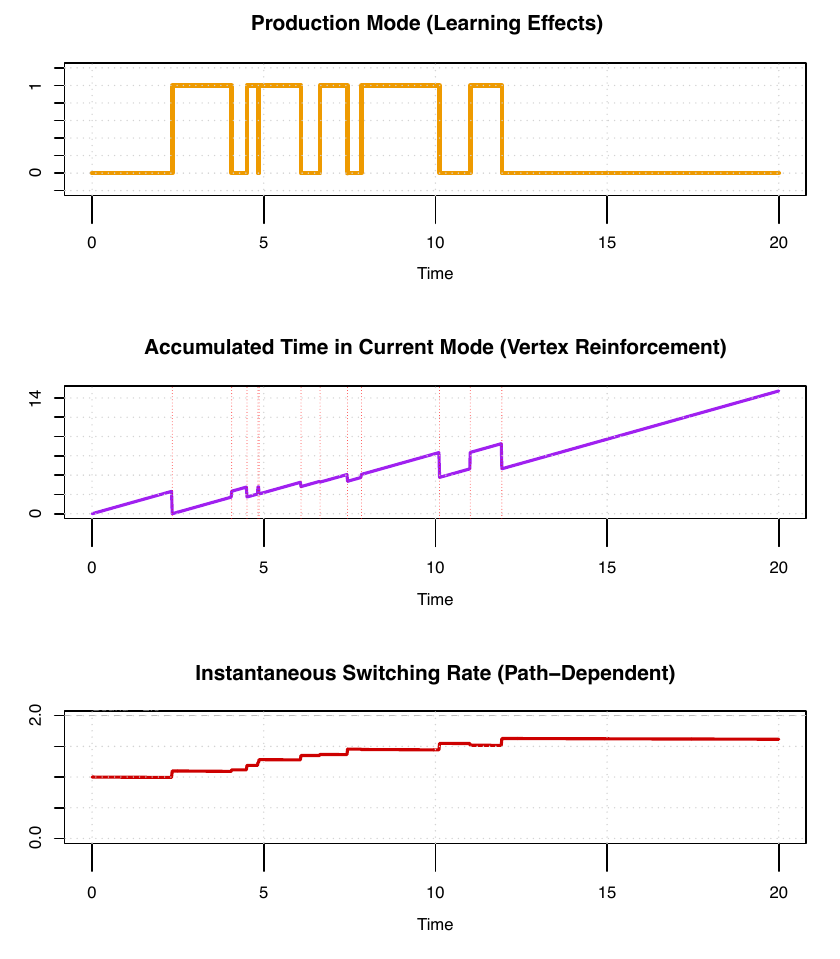}
\caption{Path reinforcement model over 20 time units showing 10 mode changes. Top: production mode evolution between standard (State 0) and optimized (State 1) with clustering of  transitions in early periods. Middle: accumulated time in current mode displaying sawtooth pattern with vertex reinforcement effects. Bottom: instantaneous transition rates increasing through edge reinforcement, bounded by $\lambda = 2.0$.}
\label{fig:reinforcement}
\end{figure}

Figure~\ref{fig:reinforcement} illustrates the self-reinforcing dynamics where early frequent  transitions (edge reinforcement) gradually increase transition propensities, while extended occupation periods create resistance through vertex inertia. The  transition rate evolves from approximately 1.0 initially to nearly 1.8 by the end, demonstrating the cumulative effect of  path reinforcement  mechanisms.

\section*{Acknowledgments}

The author acknowledges Will Allan and Giang Thu Nguyen for the initial discussions that gave birth to this line of research. The foundational ideas for this work were conceived while the author was a Research Associate at The University of Adelaide, although the theoretical framework presented herein has evolved significantly to ensure general solvability and convergence.

\bibliographystyle{apalike}
\bibliography{oscar}

\end{document}